\RequirePackage{expl3}
\documentclass[sn-mathphys,Numbered,pdftex]{sn-jnl}
\usepackage{expl3}

\usepackage{adigraph, ascmac, amsthm, amsmath, amssymb, amsfonts, bm, braket, graphicx, type1cm, textcomp, enumerate, makeidx, mathcomp, ytableau, latexsym, diagbox, amscd, verbatim, tikz, youngtab, here}

\def\Underline{\setbox0\hbox\bgroup\let\\\endUnderline}
\def\endUnderline{\vphantom{y}\egroup\smash{\underline{\box0}}\\}
\def\|{\verb|}
\usepackage[normalem]{ulem}
\usepackage[legacycolonsymbols]{mathtools}
\usepackage{mathtools}
\usepackage{hyperref}

\usepackage{multirow}%
\usepackage{mathrsfs}%
\usepackage[title]{appendix}%
\usepackage{xcolor}%
\usepackage{manyfoot}%
\usepackage{booktabs}%
\usepackage{algorithm}%
\usepackage{algorithmicx}%
\usepackage{algpseudocode}%
\usepackage{listings}%
\usetikzlibrary{decorations.shapes}
\usepackage{subcaption}
\usepackage{xspace}

\newcommand{\m}{\boldsymbol m}
\newcommand{\x}{\boldsymbol x}
\newcommand{\y}{\boldsymbol y}
\newcommand\Z{{\mathbb Z}}
\newcommand\N{{\mathbb N}}
\newcommand\Zz{{\mathbb Z_{\geq 0}}}

\newcommand\PP{{\mathcal{P}}}
\newcommand\NN{{\mathcal{N}}}
\newcommand\Pp{${\mathcal{P}}$-position}
\newcommand\Pps{${\mathcal{P}}$-positions}
\newcommand\Np{${\mathcal{N}}$-position}

\newcommand\RMi{\,{\rm i}\,}
\newcommand\RMii{\,{\rm ii}\,}

\newcommand{\+}{\oplus}

\theoremstyle{definition}
\newtheorem{theorem}{Theorem}

\newtheorem{claim}{Claim}
\newtheorem{proposition}{Proposition}
\newtheorem{corollary}{Corollary}
\newtheorem{definition}{Definition}

\newtheorem{remark}{Remark}
\newtheorem{example}{Example}

\newcommand\dyn{{\textsc{Digraph Yama Nim}\xspace}}
\newcommand\gpos{$G_{\mathrm{pos}}$(\text{POS CNF})\xspace}

\usepackage{xskak}
\usepackage{tkz-euclide}
\usepackage{skak}

\usetikzlibrary{matrix}

\nocite{*}

\begin{document}

\title[Article Title]{Digraph Yama Nim}


\author[1]{\fnm{Hiyu} \sur{Inoue}}\email{hiyuuinoue@gmail.com}
\equalcont{These authors contributed equally to this work.}

\author[2]{\fnm{Shun-ichi} \sur{Kimura}}\email{skimura@hiroshima-u.ac.jp}
\equalcont{These authors contributed equally to this work.}

\author[3]{\fnm{Hikaru} \sur{Manabe} }\email{urakihebanam@gmail.com}
\equalcont{These authors contributed equally to this work.}

\author[4]{\fnm{Koki} \sur{Suetsugu}}\email{suetsugu.koki@gmail.com}
\equalcont{These authors contributed equally to this work.}

\author*[5]{\fnm{Takahiro} \sur{Yamashita}}\email{d236676@hiroshima-u.ac.jp}

\author[6]{\fnm{Kanae} \sur{Yoshiwatari}}\email{yoshiwatari.kanae.7p@kyoto-u.ac.jp}
\equalcont{These authors contributed equally to this work.}




\affil[1,2,5]{
\orgdiv{Department of Mathematics},
\orgname{Hiroshima University}, 
\orgaddress{\street{Kagamiyama}, \city{Higashi-Hiroshima City}, \postcode{739-8526}, \state{Hiroshima}, \country{Japan}}}

\affil[3]{
\orgdiv{College of Information Science},
\orgname{University of Tsukuba}, 
\orgaddress{\street{Tennoudai}, \city{Tsukuba City}, \postcode{305-8577}, \state{Ibaraki}, \country{Japan}}}

\affil[4]{
\orgname{Gifu University}, 
\orgaddress{\street{1-1 Yanagido}, \city{Gifu City}, \postcode{501-1193}, \state{Gifu}, \country{Japan}}}

\affil*[4]{
\orgname{Waseda University}, \orgaddress{\street{513 Waseda-Tsurumaki-Cho}, \city{Sinjuku-ku}, \postcode{162-0041}, \state{Tokyo}, \country{Japan}}}

\affil*[4]{
\orgname{Osaka Metropolitan University}, 
\orgaddress{\street{3-3-138 Sugimoto}, \city{Sumiyoshi-ku}, \postcode{558-8585}, \state{Osaka}, \country{Japan}}}

\affil[6]{
\orgdiv{Graduate School of Informatics},
\orgname{Kyoto University}, 
\orgaddress{\street{Yoshida-honmachi}, \city{Sakyo-ku}, \postcode{606-8501}, \state{Kyoto}, \country{Japan}}}




\abstract{
{\sc Yama Nim} is a variant of two piles {\sc Nim}. In this ruleset, the player chosses one of the piles and removes at least two tokens from the pile. In the same move, the player adds one token to the other pile. We show the winning strategies and SG-values of this ruleset.

In addition, we introduce a generalization of {\sc Yama Nim},
named \dyn. In this ruleset, a digraph is given and there are some tokens on each vertex of the digraph. Each player, in their turn, chooses one vertex and removes at least its out-degree plus one tokens from the vertex. Furthermore, one token is added to each vertex to which a directed edge from the chosen vertex is connected. We show that the winner determination problem of \dyn~is PSPACE-complete even when the input graph
is bipartite and directed acyclic.
Despite this, there are some cases that can be solved easily and we show them.
}






\maketitle

\section{Introduction}
\label{sec:Intro}
{\sc Nim}, studied by Bouton in \cite{B02}, is one of the most popular games studied in combinatorial game theory.
In {\sc Nim}, some piles of tokens are given as a position of the game. Each player, in their turn, chooses one pile and removes an arbitrary positive number of tokens from the pile. The player who moves last is the winner.

As a generalization of {\sc Nim}, {\sc Hypergraph Nim} is studied in \cite{BGHMM24}.
In {\sc Hypergraph Nim},
a hypergraph is given and there are some tokens on each vertex of the hypergraph. Each player, in their turn, chooses one hyperedge and removes an arbitrary positive number of tokens from every vertex included in the hyperedge.
Not only the original {\sc Nim}, but also {\sc Moore's Nim} \cite{M10}, {\sc Circular Nim} \cite{DH13}, {\sc Exact-}$k$ {\sc  Nim} \cite{BGHMM18}, and {\sc Simplicial Nim} \cite{ES96} can be considered as special cases of {\sc Hypergraph Nim}.
\begin{figure}
    \centering    \includegraphics[width=1\linewidth]{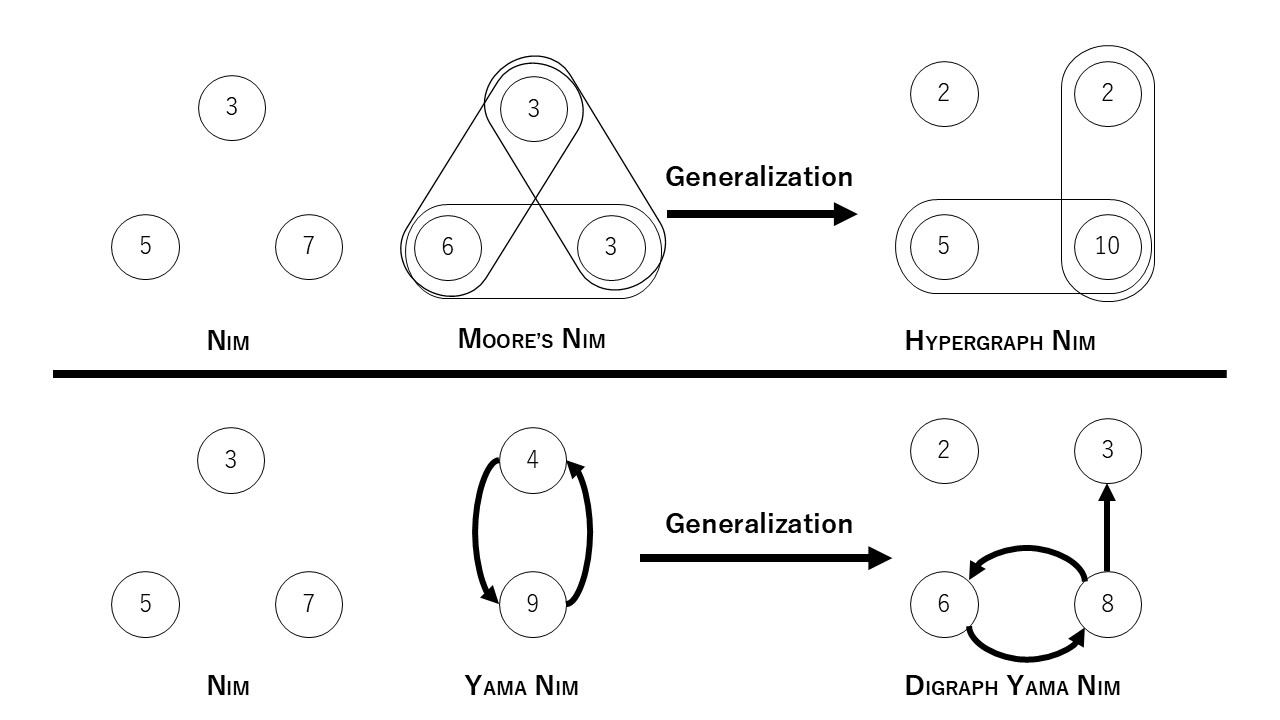}
    \caption{
    Two directions of generalizations of {\sc Nim}
    }
    \label{fig:generalization}
\end{figure}
However, in {\sc Hypergraph Nim} and its special cases, rulesets in which some tokens are added to some piles are not included. In order to consider such rulesets, we generalize {\sc Nim} in another direction by using a digraph. 

One simple ruleset in which a token is added to a pile in a move is
{\sc Yama Nim},
a ruleset introduced in the fifth author's master thesis \cite{Y23}.
In this ruleset, two piles of tokens are given as a position of the game. Each player, in their turn, chooses one of the piles and removes at least two tokens from the pile. In addition, a token is added to the other pile. The winning strategies and SG-values 
are studied in the thesis.


We consider a generalization of {\sc Nim} and {\sc Yama Nim}, called {\sc Digraph Yama Nim}. In this ruleset, a digraph is given and there are some tokens on each vertex of the digraph. Each player, in their turn, chooses one vertex and removes at least its out-degree plus one tokens from the vertex. Furthermore, one token is added to each vertex to which a directed edge from the chosen vertex is connected. 
Figure \ref{fig:generalization} shows two directions of generalizations of {\sc Nim}.

In this paper, we show PSPACE-completeness of the winner determination problem of {\sc Digraph Yama
Nim}.
Furthermore, even when the input graph is both bipartite and directed acyclic, it is still PSPACE-complete.
On the other hand, we also show that if the digraph satisfies some certain conditions, there is a method to simplify the digraph, or to solve the winner 
determination problem.

This paper is constructed as follows: In the later part of this section, we prepare 
definitions and notations. In Section \ref{sec:Yama_nim}, we introduce the original {\sc Yama Nim}. 
In Section \ref{sec:DYN}, we introduce {\sc Digraph Yama
Nim}, and show the PSPACE-completeness of winner determination problem and some solvable cases for certain conditions.




\subsection{Preliminaries}

In this paper, 
the set of all non-negative integers is denoted by $\Zz=\{0,1,2,\ldots\}$, the in-degree and out-degree of a vertex $v$ are denoted by $d_\mathrm{in}(v)$ and $d_\mathrm{out}(v)$ respectively, and let $\bigoplus_ {k=1}^{n} x_k=x_1\oplus x_2\oplus \cdots\oplus x_n$ where $\oplus$ is an exclusive OR for binary notation. We call $\oplus$ nim-sum since this is used for determining the winner of classical Nim \cite{B02}.\\

\begin{definition}[Impartial ruleset]\label{def:impartial}
  An \textbf{impartial ruleset} is a pair $\Gamma=(M, f)$, where $M$ is the set of game positions, $f: M \to Pow(M)$, where $Pow(M)$ is the set of the subsets of $M$, is the option map which sends $\m \in M$ to its set of options $f(\m) \subset M$. 
  Moreover, we assume that our game is short, namely for any $\m\in M$, there exists $N_{\m} \in \Z_{\geq0}$ such that for any game-positions sequence, starting from $\m$ ends in at most $N_{\m}$ moves.
\end{definition} 
~
\begin{remark}
    In this paper, 
    we consider an impartial ruleset under normal play (last player wins).
\end{remark}
~
\begin{definition}[Terminal position]\label{def:terminal}
  If $\m \in M$ satisfies $f(\m)=\varnothing$, then $\m$ is called a \textbf{terminal position}. We write the set of terminal positions as $\mathcal{E}$.
\end{definition}
~
~
~
\begin{definition}[\Pp~and \Np]\label{def:PpNp}
Let $\Gamma=(M,f)$ be
an impartial ruleset. We call a position $\m\in M$ is a \Pp~if the previous player has a winning strategy. We call a position $\m\in M$ is an 
\Np~if the next player has a winning strategy.
\end{definition}
~

The next proposition is a basic property of \Pp~and \Np.\\
\begin{proposition}\label{normalprop:PpNp}
  Let $\Gamma=(M,f)$ be an impartial ruleset. For all $\m\in M$,
  we can determine which player has a winning strategy as follows:
  \begin{enumerate}
  \item[(i) ] If $f(\m) = \varnothing$, then $\m$ is a \Pp.
  \item[(ii) ] 
  If there exists 
  $\m' \in f(\m)$ such that $\m'$ is a \Pp, then $\m$ is an \Np.
  \item[(iii) ] 
  If for any $\m' \in f(\m)$, $\m'$ is an \Np, then $\m$ is a \Pp.
  \end{enumerate}
  Then, any position $\m$ is either a \Pp~or an \Np.
\end{proposition}



~
Sprague and Grundy extended Bouton's theory for general impartial  
rulesets
under normal play and 
introduced
Sprague-Grundy values \cite{SP35,GR39}.
\begin{definition}[Mex]
\label{mex}
Let $X$ be a finite subset of non-negative integers.
Then the \emph{minimal excluded number $\mathrm{mex}(X)$} is 
$$\mathrm{mex}(X):=\min(\mathbb{Z}_{\geq 0} \setminus X).$$
\end{definition}
\begin{definition}[Sprague-Grundy value (SG-value)]
\label{Grundy}
Let $\Gamma = (M, f)$ be an impartial ruleset.
For any $\m \in M$, the \emph{Sprague-Grundy value} $g(\m)$ is
$$g(\m):= \mathrm{mex}(\{g({\m}') \mid \m'\in f(\m)\}).$$
\end{definition}

With regard to SG-values and winning strategies, the following proposition is known.

~
\begin{proposition}
\label{Sprague-Grundy}
Let $\Gamma = (M, f)$ be an impartial ruleset.
For any $\m \in M$, we have the following properties.
 \begin{enumerate}
    \item[\text{(1)} ] $\m$ is a \Pp~if and only if $g(\m) = 0$;
    \item[\text{(2)} ] $\m$ is an \Np~if and only if $g(\m) \neq 0$.
 \end{enumerate}
\end{proposition}
~

From Definition \ref{Grundy}, SG-values can be characterised by the following proposition. 

~
\begin{proposition}
\label{Grundyprop1}
Let $\Gamma = (M, f)$ be an impartial ruleset. 
We assume that $M$ decomposes into disjoint subsets 
$$M=\coprod_{i\in \Zz} M_i=M_0 \sqcup M_1 \sqcup M_2 \sqcup \cdots,$$
in such a way that for any $\x \in M_i$ and its option $\x' \in f(\x)$, we have $\x' \notin M_i$,
 and for any $\x \in M_k$ with $0\leq i<k$, there exists an option $\x' \in f(\x)$ such that $\x' \in M_i$. Then, $M_i$ is exactly the set of positions with their
 SG-values equal to $i$.\\
\end{proposition}

SG-values can be applied to disjunctive sum of rulesets.

~

\begin{definition}[Disjunctive Sum]
\label{def:sumgame}
Let $X=(M_1,f_1), Y=(M_2,f_2)$ be impartial rulesets. The disjunctive sum $X + Y$ is an impartial ruleset $(M,f)$ where the set of game positions is $M=M_1 \times M_2$ and for any $\x\in M_1$ and for any $\y \in M_2$, the option map $f(\x,\y)$ is
\begin{eqnarray*}
       f(\x,\y) = \{(\x', \y)\in M\, |\, \x' \in f_1(\x)\}  \cup \{(\x, \y')\in M\, |\, \y' \in f_2(\y)\}.
\end{eqnarray*}
\end{definition}
~
\begin{theorem}[Sprague-Grundy Theorem \cite{SP35,GR39}]
\label{thm:Sprague-Grundygeneral}
Let $X + Y=(M,f)$ be an impartial ruleset. For any $\x\in M_1$ and for any $\y \in M_2$, we have
$$g(\x+\y)= g(\x) \oplus  g(\y).$$

\end{theorem}
\section{Yama Nim}
\label{sec:Yama_nim}
In this section, we introduce the game {\sc Yama Nim} studied in \cite{Y23}. Since the thesis is written in Japanese, we show not only the statement of theorems but also the proofs of them in English.
In particular, in this paper, we show
a different way to show SG-values of {\sc Yama Nim}.

In this paper, 
we simply write $f(x,y)$ for $f((x,y))$ and $g(x,y)$ for $g((x,y))$.\\
\begin{definition}[{\sc Yama Nim}]
\textbf{\sc Yama Nim} is an impartial ruleset $\Gamma = (\Z_{\geq 0} \times \Z_{\geq 0}, f)$ such that
\begin{eqnarray*}
\label{Eqn:f_function_Yama}
       f(x,y) &=& \{(x-i, y+1)\in M\, |\, 2\le i\le x\} \\
                & & \cup \{(x+1, y-i)\in M\, |\,  2\le i \le y\}.
\end{eqnarray*}
\end{definition}

\begin{theorem}
\label{YamaPps}
The set of \Pps~of {\sc Yama Nim} $Y$ is
$$Y=\{(x, y)\in M\, |\, |x-y| \leq 1\}.$$
\end{theorem}

\begin{proof}
  We need to show that
 
 \begin{enumerate}
    \item[\text{(\RMi)} ] For any $(x,y) \in Y$, its option $(x',y')$ is not in $Y$;
    \item[\text{(\RMii)} ] For any $(x,y) \notin Y$, there is an option $(x',y') \in Y$.
  \end{enumerate}
  
For (\RMi), we take $(x,y) \in Y$.
Without loss of generality, we may move to $(x-i,y+1)$ with $i \geq 2$, then 
\begin{equation*}
       |(x-i)-(y+1)| = |(x - y)-(i+1)| \geq |i+1|-|x-y| \geq 2,
\end{equation*}
hence $(x-i,y+1) \notin Y$. 

For (\RMii), we take $(x,y) \notin Y$ and without loss of generality,
we may assume  $y - x \geq 2$, then we can move to
$$ (x+1,y-(y-x))=(x+1,x)\in Y.$$
\end{proof}

Next, we 
describe the SG-value of {\sc Yama Nim}.

\begin{theorem}
\label{yama1} The SG-value of {\sc Yama Nim} is
$$g(x,y)=\left\{
\begin{array}{cc} 
0& (|x-y|\le 1)\\ 
\min(x, y)+1 & (|x-y|>1).
\end{array}
\right.$$
\end{theorem}
\begin{proof} 
Let $Y_0=\left\{ (x, y) \,\big|\, |x-y|\le 1 \right\}$ and, for $i>0$, let $Y_i=\left\{(x, y) \,\big|\, |x-y| > 1, \min(x, y)+1=i \right\}$.
By Proposition \ref{Grundyprop1}, we need to show that
 \begin{enumerate}
    \item[\text{(\RMi)} ] For any $(x,y) \in Y_i~~(i\geq0)$, its option $(x',y')$ is not in $Y_i$;
    \item[\text{(\RMii)} ] If $i<k$, for any $(x,y) \in Y_{k}$, there is an option $(x',y') \in Y_i$. 
 \end{enumerate}

The case $i=0$ was shown in Theorem \ref{YamaPps}, so we may assume $i>0$.

For (\RMi), we take $(x, y)\in Y_i$, and assume that $(x', y')$ is an option of $(x,y)$.
We may assume $x<y$ with $y-x>1$.  
Then we have either $x' \leq x-2$ and $y'=y+1$, or $x'=x+1$ and $y'\leq y-2$. In the first case, we have $\min(x', y') < \min(x, y+1)$, hence $(x', y')\notin Y_i$.
In the second case, if $\min(x', y')=\min(x, y)$, then we need to have $y'=x$, but $(x', y')=(x+1,x)\in Y_0$, not in $Y_i$.

For (\RMii), we take $(x, y)\in Y_k$ with $k>i$. 
We may assume $k-1=x<y$ with $y-x>1$.
Then, we can move to $(x+1, y-(y-(i-1))=(x+1,i-1) \in Y_i$ because
$$y-(i-1)>y-(k-1)=y-x>1,$$
$$(x+1)-(i-1)=(k-1+1)-(i-1) = (k-i)+1 \geq 2.$$
\end{proof}

{\sc Yama Nim} under mis\`ere play is also considered in \cite{Y23}.


\section{Digraph Yama Nim}
\label{sec:DYN}

In this section, firstly we introduce \dyn. 

\begin{definition}[\dyn]
The rules of \textbf{\dyn} is as follows.

There is a digraph $G = (V, E), V = \{v_1, \cdots, v_n\}$
and tokens are placed on each vertex.
The player chooses a vertex $v_i \in V$ and takes at least $d_\mathrm{out}(v_i)+1$ tokens.
At the same time, the player adds one token to all out-neighbors of
$v_i$.

In other words, \dyn~is a ruleset $\Gamma=(M,f)$ such that

$$M = (\Zz)^n,$$
$$f(x_1, \cdots , x_n) = \left\{(x'_{1}, \ldots, x'_{n}) \middle | 
    \begin{array}{l}
    \text{for an}~i\in \{1,2,\cdots n\}, 0\le x'_{i}<x_{i} - d_\mathrm{out}(v_i), \\
      \text{for any}~j~\text{which satisfies}~(v_i, v_j)\in E, x'_{j} = x_{j} + 1, \\
      otherwise, x'_{k} = x_{k}
     \end{array}
    \right\},$$
    where $x_i$ and $x'_i$ are the numbers 
    of tokens on vertex $v_i$.\\
\end{definition}

\begin{remark}
    If a vertex does not have any outgoing edges,
    we can take an arbitrary positive number of tokens, but up to the number of tokens on the vertex, from that vertex.\\
\end{remark}

{\dyn} is a generalization of classical {\sc Nim} and {\sc Yama Nim}.\\
\begin{example}
\label{rem:nim}
    Let $G=(V,E)$ be a graph on $n$ vertices with no directed edges. Then, \dyn~on $G$ represents {\sc Nim} with $n$ piles~(See Figure \ref{fig:generalization}).\\
\end{example}

\begin{example}
\label{rem:yama}
 Let $G = (V, E)$ be a digraph, where $V = \{v_1, v_2\}$ and $E = \{(v_1, v_2), (v_2, v_1)\}$.
 Then, 
  \dyn~on $G$ represents {\sc Yama Nim}
 ~(See Figure \ref{fig:generalization}).
\end{example}


\subsection{PSPACE-completeness of \dyn}
\label{subsec:Complex}



\medskip

In this section, we prove that the winner determination problem of \dyn\ is PSPACE-complete.
By the convention, the name of a ruleset is also used as the name of the decision problem of determining the winner.

PSPACE-completeness of \dyn\ is proved by a reduction from \gpos.
\gpos\ is played on a positive CNF formula $F$.
Players ($P_1$ and $P_0$) alternately select a variable of $F$ which has not been selected ever. 
The variables selected by $P_1$ are assigned the value 1, and those selected by $P_0$ are assigned the value 0.
After all variable assignments have been fixed,  $P_1$ wins if $F = 1$, and $P_0$ wins if $F = 0$.
\gpos\ is known to be PSPACE-complete~\cite{S78}.


\medskip

\begin{theorem}
    \dyn\ is PSPACE-complete even when the input graph is bipartite and directed acyclic.
\end{theorem}

\medskip


\begin{proof}
We construct a reduction graph $G$ from $F$ such that the first player wins in \dyn\ on $G$ if and only if $P_1$, the first player, wins in \gpos\ on $F$.
Let $X_1, \ldots, X_n$ be the variables and $C_1, \ldots, C_m$ be the clauses of $F$.
We define $|C_j|$ as the number of variables in $C_j$.

$G$ consists of subgraphs $G_1$, $G_2$, and $G_3$ as shown in Figure \ref{fig:pspace-c}.
$G_1$ contains vertices $x_1, \ldots, x_n$, which vertices are pairwise non-adjacent.
For each $1 \leq i \leq n$, a directed path $Q_i$ extends from $x_i$, where $|Q_i| = 3$. We define the first vertex of the path $Q_i$ as $y_i$.
If $n$ is odd, add an isolated vertex called the \emph{parity vertex} and extend a directed edge from the end of $Q_1$.
%
$G_2$ contains vertices $c_1, \ldots, c_m$, which vertices are pairwise non-adjacent.
For each $1 \leq j \leq m$, $|C_j| - 1$ directed paths of length 2 extend from $c_j$.
Every $c_j$ also have an outgoing edge to a vertex $f_j$.
In $G_3$, there is a vertex $w$, and $m-1$ directed paths of the length 2 extend from $w$.
We next describe the edges among $G_1$, $G_2$, and $G_3$.
There is a directed edge $(x_i,c_j)$ if in $F$ a variable $X_i$ is included in $C_j$, and every $c_j$ has an outgoing edge to $w$.

The number of tokens on each vertex is as follows:
\begin{itemize}
    \item For $1 \leq i \leq n$, $x_i$ has $2(d_\mathrm{out}(x_i)+1)$ tokens.
    \item For $1 \leq j \leq m$, $c_j$ has one token.
    \item If the parity vertex exists, it has one token.
    \item None of the other vertices have any tokens. 
\end{itemize}


 \begin{figure}[htbp]
  \centering
  \includegraphics[width=1\linewidth]{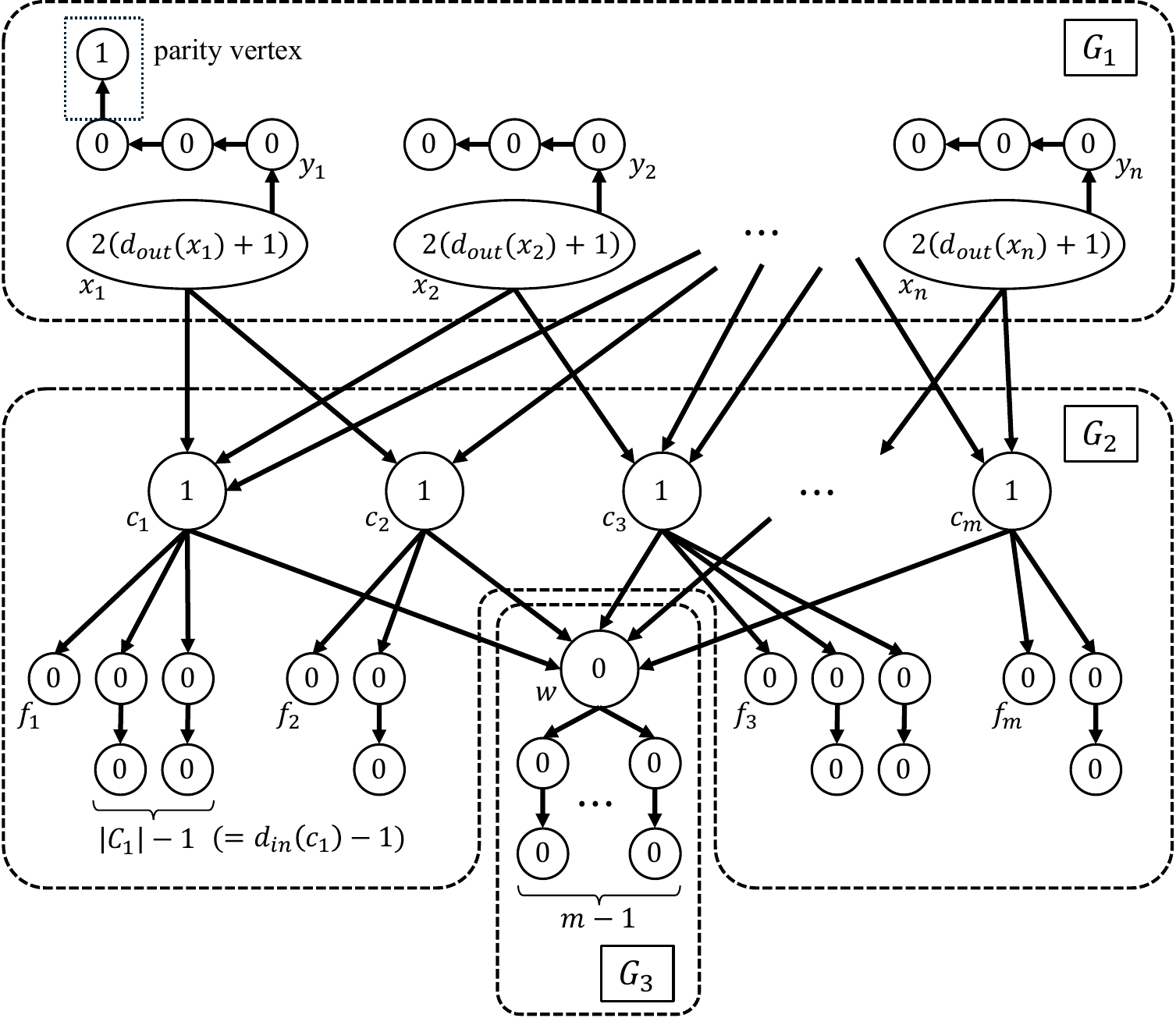}
  \caption{A reduction graph $G$.}
  \label{fig:pspace-c}
 \end{figure} 

\begin{claim}\label{claim_even}
In every gameplay, both the total number of moves made within $G_1$ and the total number of moves made within $G_2$ are even.
\end{claim}

\begin{proof}
%
In $G_1$,  
every $x_i$ can be played at least once.  
If $d_\mathrm{out}(x_i) + 1$ tokens are removed from $x_i$, $x_i$ can be played once more.  
Let $t_x$ denote the number of vertices $x_i (1 \leq i \leq n)$ that are played twice.  
When $x_i$ is played twice, $y_i$ holds two tokens, so $y_i$ can be played once.  
After playing $y_i$, no vertex on $Q_i$ can be played.  
Therefore, if $n$ is even, the total number of moves is $n + 2t_x$, which is even.  
If $n$ is odd, considering one move on the parity vertex, the total number of moves in $G_1$ is $n + 1 + 2t_x$, which is also even.


For $G_2$, let $t_c$ denote the number of vertices $c_j (1 \leq j \leq m)$ that are played during the game.  
Since 
each $x_i$ can send at most two tokens, every $c_j$ can be played at most once. 
When $c_j$ is played, one token is sent to $f_j$, and thus $f_j$ can be played once.  
Therefore, the total number of moves made within $G_2$ is $2t_c$.
\end{proof}

Let $P_\mathcal{N}$ and $P_\mathcal{P}$ denote the first and second players to move in the initial position of \dyn, respectively.
By Claim \ref{claim_even}, the number of moves made in $G_3$ determines the parity of total moves in a gameplay and the winner. 
Since $d_\mathrm{out}(w) = m - 1$, $P_\mathcal{N}$ can move at $w$ if tokens are sent from all of $c_1, \ldots, c_m$.


Suppose that $P_1$ has a winning strategy in \gpos\ on $F$.
Then, $P_\mathcal{N}$ selects a vertex $x_i$ based on the winning strategy of $P_1$ and removes $d_\mathrm{out}(x_i) + 1$ tokens from $x_i$ in  $G$.
Since $x_i$ still has $d_\mathrm{out}(x_i) + 1$ tokens remaining, $x_i$ is played again.
As a result, each $c_j$ connected to $x_i$ ends up with at least $|C_j| + 2$ tokens in total.
Since $d_\mathrm{out}(c_j) = |C_j| + 1$, the vertex $c_j$ becomes playable.
Using the winning strategy on $F$, $P_\mathcal{N}$ can make all $c_j$ playable, and thus $w$ will accumulate $m$ tokens.
Therefore, by playing on $w$, the total number of moves becomes odd, and $P_\mathcal{N}$ wins in \dyn.


Suppose that $P_0$ has a winning strategy in \gpos\ on $F$.
Then, $P_\mathcal{P}$ selects a vertex $x_i$ based on the winning strategy of $P_0$ and removes at least $d_\mathrm{out}(x_i) + 2$ tokens from $x_i$ in $G$.
Since $x_i$ now has at most $d_\mathrm{out}(x_i)$ tokens, it can no longer be played.
Using the winning strategy in $F$, $P_\mathcal{P}$ ensures that some vertex $c_j$ ends up with exactly $|C_j| + 1$ tokens.
Since $d_\mathrm{out}(c_j) = |C_j| + 1$, the vertex $c_j$ remains unplayable.
Because $d_\mathrm{out}(w) = m - 1$, if even one $c_j$ fails to send a token, then $w$ cannot be played.
By Claim \ref{claim_even}, since the total number of moves in $G$ is even, $P_\mathcal{P}$ wins in \dyn.

Finally, we verify that $G$ is a bipartite and directed acyclic graph.  
We consider a partition of vertices of $G$ into sets $A$ and $B$. 
The following vertices are assigned to $A$:
$x_i$ and the second vertex in $Q_i$ for $1 \leq i \leq n$, the parity vertex, $f_j$ and the first vertex of all paths of length 2 which extend from $c_j$ for $1 \leq j \leq m$, the last vertex of all paths of length 2 which extend from $w$, and $w$.
All other vertices are in $B$.
Since this partition ensures that there are no edges within both $A$ and $B$, $G$ is a bipartite graph.

Moreover, $G_1$, $G_2$, and $G_3$ contain no cycles, and vertices in $G_1$ have edges only to $G_2$, while vertices in $G_2$ have edges only to $G_3$.  
Since vertices in $G_3$ do not have edges to any vertex in $G_1$ or $G_2$, the graph $G$ contains no cycles.

We conclude the proof by presenting a concrete example in Figure~\ref{fig:example}.

 \begin{figure}[htbp]
  \centering
\includegraphics[width=0.8\linewidth]{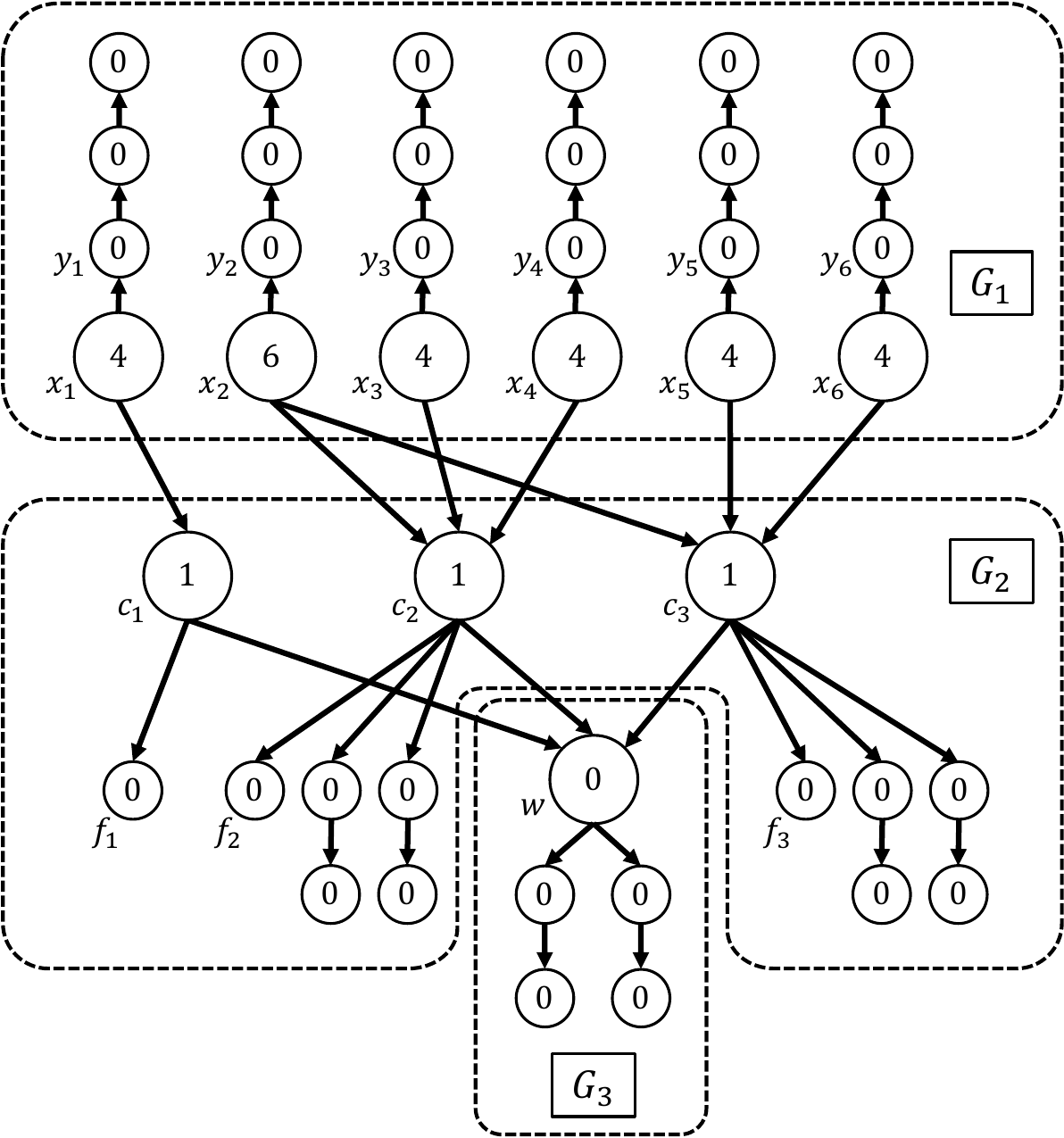}
  \caption{The reduction graph corresponding to $F=x_1 \land(x_2 \lor x_3 \lor x_4) \land (x_2 \lor x_5 \lor x_6)$.}
  \label{fig:example}
 \end{figure} 
 
%
\end{proof}

\subsection{Solvable cases and simplification methods on \dyn}
\label{subsec:SSDYN}
In the previous subsection, we showed the PSPACE-completeness of \dyn. Despite this, there are some cases that can be solved easily, or simplified to a smaller position.\\

\begin{definition}
Let $G_4=(V,E)$ be a digraph satisfying the following conditions:
    \begin{itemize}
        \item The vertex set $V$ can be partitioned into $V=V_1\cup W_1$ such that $V_1\cap W_1=\varnothing$.
        \item For any $v\in V_1$ and any $w\in W_1$, there exists a unique directed edge $(v,w)\in E$.
        \item For any $w\in W_1$, $d_\mathrm{out}(w)=0$. 
    \end{itemize}
    \begin{figure}[tb]
        \centering
        \includegraphics[width=1\linewidth]{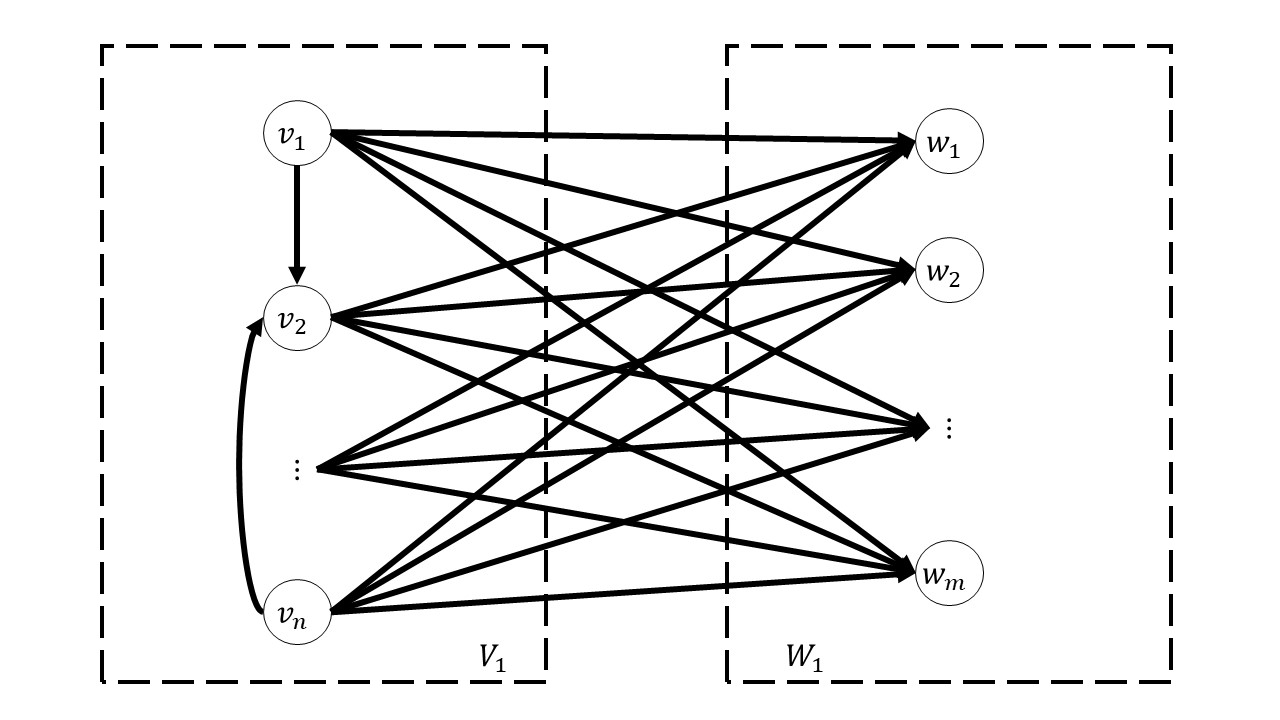}
        \caption{Graph $G_4$}
        \label{fig:simplify}
    \end{figure}
    Figure \ref{fig:simplify} shows the construction of graph $G_4$.
    Let $V_1=\{v_1,\dots, v_n\}$ and $W_1=\{w_1,\dots, w_m\}$ and a position in the game is represented as $(x_1,\dots, x_n, y_1, \dots, y_m)$, where $x_i$ is the number of tokens on vertex $v_i$ and $y_j$ is that on vertex $w_j$.\\
\end{definition}

We separate into two cases depending on the parity of $\lvert W_1\rvert$.\\
\begin{theorem}
\label{Thm:DYNodd}
    
If $\lvert W_1\rvert$ is odd, for \dyn~played on $G_4$, the SG-value $g(x_1,\dots, x_n, y_1, \dots, y_m)$
is 
$$g(x_1,\dots, x_n, y_1, \dots, y_m)=\bigoplus_{j=1}^{m}y_j.$$
        In particular, the position $(x_1,\dots, x_n, y_1, \dots, y_m)$ is a \Pp~if and only if $\bigoplus_{j=1}^{m}y_j=0$.
        
\end{theorem}
\begin{proof}
We separate the set of positions depending on the XOR of the numbers of piles on each vertex in $W_1$ 
as follows:
$$S_i = \left\{(x_1,\ldots, x_n, y_1,\ldots,y_m)\middle| \bigoplus_{j=1}^{m}y_j = i\right\}.$$
We need to show that
    \begin{enumerate}
        \item[\text{(\RMi)} ] For any $(\textbf{x},\textbf{y}) \in S_i$, its option $(\textbf{x}',\textbf{y}')$ is not in $S_i$;
        \item[\text{(\RMii)} ] If $i<k$, for any $(\textbf{x},\textbf{y}) \in S_k$, there is an option $(\textbf{x}',\textbf{y}') \in S_i$.
  \end{enumerate}
    For (\RMi), we take $(\textbf{x}, \textbf{y})=(x_1,\dots, x_n,y_1,\dots, y_m)\in S_i$ and assume we move to $(\textbf{x}',\textbf{y}')=(x_1',\dots,x_n',y_1',\dots, y_m')$. 

\begin{itemize}
\item When the move is made on a vertex in $V_1$,
since $\bigoplus_{j=1}^{m} y_j = i$, the parity of the number of $y_j$ (for $j\in\{1,\dots,m\}$) with odd is the same as the parity of $i$ and the parity of the number of $y_j$ (for $j\in\{1,\dots,m\}$) with even is not the same as the parity of $i$.
Consequently, the parity of the number of $y_j+1$ (for $j\in\{1,\dots,m\}$) with odd is not the same as the parity of $i$. Thus, $\bigoplus_{j=1}^{m} (y_{j}')=\bigoplus_{j=1}^{m} (y_{j}+1) \neq i$, so $(\textbf{x}',\textbf{y}')\not\in S_i$.
\item When the move is made on a vertex in $W_1$, by the properties of nim-sum, $\bigoplus_{j=1}^{m} (y_{j}') \neq i$, so $(\textbf{x}',\textbf{y}')\not \in S_i$.
\end{itemize}

For (\RMii), we take $(\textbf{x}, \textbf{y})=(x_1,\dots, x_n,y_1,\dots, y_m)\in S_k$ with $i < k$. Then, from the definition of $S_k$, $\bigoplus_{j=1}^{m} y_j = k$. By the properties 
of nim-sum, there exists some $y_j^{'} < y_j$ such that $y_{1}\+ \dots \+y_{j-1}\+ y_j^{'}\+y_{j+1}\+\dots \+ y_{m}=i$. This move transforms the position into an element of $S_i$.

\end{proof}


\begin{theorem}\label{Yama main}
\label{Thm:DYNeven}
Let $\lvert W_1\rvert$ be even.
In order to analyze this case,
we consider another ruleset on the same digraph:
 A player chooses $v_i\in V_1$ and takes at least $d_\mathrm{out}(v_i)+1$ tokens from $v_i$, adds one token to each $v_j\in V_1$ where $(v_i, v_j)\in E$, 
and add no token to $w_k\in W_1$ where $(v_i,w_k)\in E$. Note that for any $v_i\in V_1$, $d_\mathrm{out}(v_i)=|\{x\in (V_1\cup W_1)\mid (v_i, x)\in E\}|$. 
That is, tokens are not added to and taken from vertices in $W_1$, but when tokens are taken from a vertex $v_i \in V_1$, vertices in $W_1$ are taken into account for counting $d_{\rm out}(v_i)$.
We write $S_{\PP}$ for the set of $\PP$-positions of this game and $S_{\NN}$ for the set of $\NN$-positions of this game.


Then, for original \dyn~played on $G_4$, the set of \Pps~is $S=S_1\cup S_2$ where
    \begin{eqnarray*}
        S_1 &=& \left\{(x_1,\dots, x_n,y_1,\dots, y_m)\middle| (x_1,\dots, x_n)\in S_{\PP} \text{ and } \bigoplus_{j=1}^{m} y_j = 0 \right\},\\
        S_2 &=& \left\{(x_1, \dots, x_n,y_1,\dots, y_m)\middle| (x_1,\dots, x_n)\in S_{\NN} \text{ and } \bigoplus_ {j=1}^{m} y_j = 1 \right\}.
    \end{eqnarray*}
\end{theorem}

\begin{remark}
    We get the same $S_{\PP}$ and $S_{\NN}$ when we consider the following \dyn, which has another graph
    : We replace all elements of $W_1$ with a vertex $w_0$
    . Additionally, we add new vertices, more than the sum of tokens placed on $G$, and add edges extend from $w_0$ to these newly added vertices. Then, $S_{\PP}$ is the set of $\PP$-positions and $S_{\NN}$ is the set of $\NN$-positions of this \dyn. Since it is not possible to move the tokens 
    at vertex $w_0$, this ruleset becomes equivalent to that described above.
\end{remark}
\begin{proof}
  We need to show that
 
    \begin{enumerate}
        \item[\text{(\RMi)} ] For any $(\textbf{x},\textbf{y}) \in S$, its option $(\textbf{x}',\textbf{y}')$ is not in $S$.
        \item[\text{(\RMii)} ] For any $(\textbf{x},\textbf{y}) \notin S$, there is an option $(\textbf{x}',\textbf{y}') \in S$.
  \end{enumerate}
    For (\RMi), we take $(\textbf{x}, \textbf{y})=(x_1,\dots, x_n,y_1,\dots, y_m)\in S$.
    \begin{itemize}
        \item  If $(\textbf{x},\textbf{y})\in S_1$, assume we move to $(\textbf{x}',\textbf{y}')=(x_1',\dots,x_n',y_1',\dots, y_m')$.
        \begin{itemize}
            \item When the move is made on a vertex in $V_1$, by the definition of $S_{\PP}$, $\textbf{x}'\in S_{\NN}$. Therefore, $(\textbf{x}',\textbf{y}')\not\in S_1$. Also, since $\bigoplus_{j=1}^{m} y_j = 0$, the number of $y_j$ (for $j\in\{1,\dots,m\}$) with even is even. Consequently, the number of $y_j+1$ (for $j\in\{1,\dots,m\}$) with odd is even.
            Thus, $\bigoplus_{j=1}^{m} (y_{j}')=\bigoplus_{j=1}^{m} (y_{j}+1)$ is even, so $(\textbf{x}',\textbf{y}')\not\in S_2$. Hence, $(\textbf{x}',\textbf{y}')\not \in S$.
            \item When the move is made on a vertex in $W_1$, by the properties of nim-sum, $(\textbf{x}',\textbf{y}')\not \in S_1$. Since $\textbf{x}'=\textbf{x}$, we have $\textbf{x}'\in S_{\PP}$, so $(\textbf{x}',\textbf{y}')\not\in S_2$. Therefore, $(\textbf{x}',\textbf{y}')\not\in S$.
        \end{itemize}
        \item If $(\textbf{x},\textbf{y})\in S_2$, assume we move to $(\textbf{x}',\textbf{y}')=(x_1',\dots,x_n',y_1',\dots, y_m')$.
        \begin{itemize}
            \item When the move is made on a vertex in $V_1$, we show that $\bigoplus_{j=1}^{m} (y_{j}+1) \ge 3$. Since $\bigoplus_{j=1}^{m} y_{j}=1$, the number of $y_j$ (for $j\in\{1,\dots,m\}$) with odd is odd and the number of $y_j$ with even is odd. For such odd $y_j$, the second least significant bit of $y_j$ and $y_j+1$ in binary representation are different. That is, 
            \[\bigoplus_{\substack{1\le i\le m\\ y_i:\text{odd}}} y_{j} = \dots 01_{(2)} \text{ and } \bigoplus_{\substack{1\le i\le m\\ y_i:\text{odd}}} (y_{j}+1) = \dots 10_{(2)}\] or \[\bigoplus_{\substack{1\le i\le m\\ y_i:\text{odd}}} y_{j} = \dots 11_{(2)} \text{ and } \bigoplus_{\substack{1\le i\le m\\ y_i:\text{odd}}} (y_{j}+1) = \dots 00_{(2)}\] hold.
            Note that for the former case, since $\bigoplus_{j=1}^{m} y_{j}=1$, \[\bigoplus_{\substack{1\le i\le m\\ y_i:\text{even}}} y_{j} = \dots 00_{(2)}\] holds and for the later case, \[\bigoplus_{\substack{1\le i\le m\\ y_i:\text{even}}} y_{j} = \dots 10_{(2)}\] holds. In addition, 
            for even $y_j$, the second least significant bit of $y_j$ and $y_j+1$ in binary representation are the same. Therefore, for the former case, \[\bigoplus_{\substack{1\le i\le m\\ y_i:\text{even}}} (y_{j} +1)= \dots 01_{(2)}\] holds and for the later case, \[\bigoplus_{\substack{1\le i\le m\\ y_i:\text{even}}} (y_{j}+1) = \dots 11_{(2)}\] holds.
             Thus, \[\bigoplus_{j=1}^{m} (y_{j}+1)=\bigoplus_{\substack{1\le i\le m\\ y_i:\text{odd}}} (y_{j}+1)\+ \bigoplus_{\substack{1\le i\le m\\ y_i:\text{even}}} (y_{j}+1)=\dots 10_{(2)}\+ \dots 01_{(2)}=\dots 11_{(2)}\] or \[\bigoplus_{j=1}^{m} (y_{j}+1)=\bigoplus_{\substack{1\le i\le m\\ y_i:\text{odd}}} (y_{j}+1)\+ \bigoplus_{\substack{1\le i\le m\\ y_i:\text{even}}} (y_{j}+1)=\dots 00_{(2)}\+ \dots 11_{(2)}=\dots 11_{(2)}\]  holds. In any case, the last two bits of $\bigoplus_{j=1}^{m} (y_{j}+1)$ in binary representation are $11$, which implies $\bigoplus_{j=1}^{m} (y_{j}+1) \ge 3$. Therefore, $(\textbf{x}',\textbf{y}')\not\in S$.
            \item When the move is made on a vertex in $W_1$, we have $\textbf{x}'\in S_{\NN}$ since $\textbf{x}'=\textbf{x}$. Thus $(\textbf{x}',\textbf{y}')\not\in S_1$. By the properties of nim-sum, $(\textbf{x}',\textbf{y}')\not \in S_2$. Therefore, $(\textbf{x}',\textbf{y}')\not\in S$.
        \end{itemize}
    \end{itemize}

~

    For (\RMii), we take $(\textbf{x}, \textbf{y})=(x_1,\dots, x_n,y_1,\dots, y_m)\not\in S$.
    \begin{itemize}
        \item For the case $(x_1,\dots, x_n)\in S_{\PP}$, $\bigoplus_{j=1}^{m} y_j \neq 0$. By the properties of nim-sum, 
        there exists some $y_j^{'} < y_j$ such that $y_{1}\+ \dots \+y_{j-1}\+ y_j'\+y_{j+1}\+\dots \+ y_{m}=0$. This move transforms the position into an element of $S_1$.

        \item For the case $(x_1,\dots, x_n)\in S_{\NN}$, $\bigoplus_{j=1}^{m} y_j \neq 1$.
        \begin{itemize}
            \item If $\bigoplus_{j=1}^{m} y_j \ge 2$, by the properties of nim-sum, there exists some $y_j^{'} < y_j$ such that $y_{1}\+ \dots \+y_{j-1}\+ y_j^{'}\+y_{j+1}\+\dots \+ y_{m}=1$. This move transforms the position into an element of $S_2$.
            \item If $\bigoplus_{j=1}^{m} y_j =0$:
                \begin{itemize}
                    \item If all $y_j$ (for $j\in\{1,\dots,m\}$ ) are even, there exists a position $\textbf{x}'$ such that $\textbf{x}\to \textbf{x}'\in S_{\PP}$ because $\textbf{x}\in S_{\NN}$. Note that $\bigoplus_{j=1}^{m} (y_j+1)=\bigoplus_{j=1}^{m} (y_j\+ 1)=\bigoplus_{j=1}^{m} y_j=0$. Thus, the resulting position is an element of $S_1$.
                    \item If there exists an odd $y_j$ (for $j\in\{1,\dots,m\}$), let $y_l$ be an odd number. By removing one token from $w_l$, we can transform the position into an element of $S_2$.
                \end{itemize}
        \end{itemize}
    \end{itemize}
\end{proof}

From these theorems, despite that the general problem is PSPACE-complete, we can solve the winner determination problems for many cases of \dyn. Here, we show some examples as follows.

\begin{corollary}\label{dia2}
Let $G_5=(V,E)$ be a digraph with $V=\{v_1,v_2,w_1,w_2\}$ and $E=\{(v_1,v_2), (v_1,w_1), (v_1,w_2), (v_2,w_1), (v_2, w_2)\}$.
We consider \dyn~on $G_5$, in other words, we consider a ruleset $\Gamma=(M,f)$ such that
  \begin{eqnarray*}
  M &=& (\Zz)^{4},\\
  f(x,y,z_1,z_2)&=& \{(x-i,y+1,z_1+1,z_2+1) \mid 4 \leq i \leq x \} \\
           &&\cup \{(x,y-i,z_1+1,z_2+1) \mid 3 \leq i \leq y_1 \} \\
           &&\cup \{(x,y,z_1-i,z_2) \mid 1 \leq i \leq z_1 \} \\
           &&\cup \{(x,y,z_1,z_2-i) \mid 1 \leq i \leq z_2 \},
  \end{eqnarray*}
      where $x$ is the number of tokens on vertex $v_1$, $y$ is that on vertex $v_2$, and $z_i$ is that on vertex $w_i$, respectively.
    
    Then, the set of \Pps~of this ruleset is $S=S_1\cup S_2$ where 
    \begin{equation*}
        S_1 = \left\{(x, y, z_1, z_2)\in (\mathbb{Z}_{\ge 0})^4\middle| \left(2\left\lfloor\frac{x}{4}\right\rfloor\le y \le 2\left\lfloor\frac{x}{4}\right\rfloor + 2\right) \text{ and } (z_1 \+ z_2 = 0)\right\},
    \end{equation*}
    \begin{equation*}
        S_2 = \left\{(x, y, z_1, z_2)\in (\mathbb{Z}_{\ge 0})^4\middle| \left(y < 2\left\lfloor\frac{x}{4}\right\rfloor \text{ or } 2\left\lfloor\frac{x}{4}\right\rfloor + 2 < y\right) \text{ and } ( z_1 \+ z_2 = 1 )\right\}.
    \end{equation*}
\end{corollary}
~

Figure \ref{fig:g5} shows the digraph $G_5$.
\begin{figure}
    \centering
    \includegraphics[width=1\linewidth]{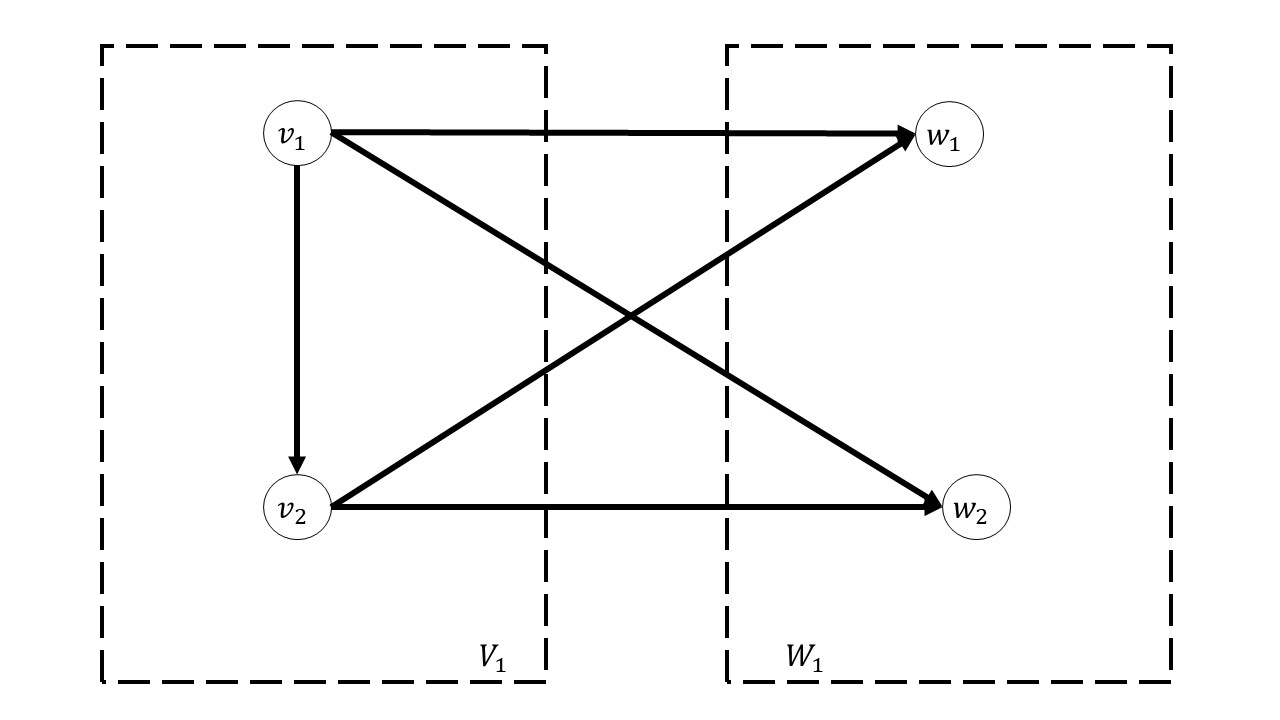}
    \caption{Digraph $G_5$}
    \label{fig:g5}
\end{figure}
\begin{proof}
    In order to apply Theorem \ref{Thm:DYNeven}, we consider a combinatorial game on $G_5$ allowing the following moves:
    \begin{itemize}
        \item Remove at least $4$ tokens from $v_1$ and add $1$ token to $v_2$.
        \item Remove at least $3$ tokens from $v_2$.
    \end{itemize}
    For this combinatorial game, if we show that $S_{\mathcal{P}}= \{2\left\lfloor{\frac{x}{4}}\right\rfloor\le y \le 2\left\lfloor{\frac{x}{4}}\right\rfloor + 2\}$ 
    coincides with the set of all $\PP$-positions, then the original theorem follows from Theorem \ref{Yama main}.
    
    We need to show that
 
    \begin{enumerate}
        \item[\text{(\RMi)} ] For any $(x,y) \in S_{\mathcal{P}}$, its option $(x',y')$ is not in $S_{\mathcal{P}}$.
        \item[\text{(\RMii)} ] For any $(x,y) \notin S_{\mathcal{P}}$, there is an option $(x',y') \in S_{\mathcal{P}}$.
    \end{enumerate}
    For (\RMi), we take $(x, y)\in S_{\mathcal{P}}$ and we write $(x',y')$ for its option.
    \begin{itemize}
        \item When a move is made on $x$, we note that; 
        \begin{equation*}
            \left\{ \,
            \begin{aligned}
                & x' = x-i\quad (i\in \{4,5,\dots, x\}) \\
                & y' = y + 1.
            \end{aligned}
            \right.
        \end{equation*}
        Since
        \begin{equation*}
            2\left\lfloor{\frac{x'}{4}}\right\rfloor = 2\left\lfloor{\frac{x - i}{4}}\right\rfloor \le 2\left\lfloor{\frac{x- 4}{4}}\right\rfloor =2\left\lfloor{\frac{x}{4}}\right\rfloor - 2,
        \end{equation*}
        we have $2\left\lfloor{\frac{x'}{4}}\right\rfloor + 2 \le 2\left\lfloor{\frac{x}{4}}\right\rfloor$.
        Also we know that $2\left\lfloor{\frac{x}{4}}\right\rfloor \le y < y'$, so $2\left\lfloor{\frac{x'}{4}}\right\rfloor + 2 < y'$, which implies $(x',y')\not \in S_{\mathcal{P}}$.
        \item When a move is made on $y$, we note that:
        \begin{equation*}
            \left\{ \,
            \begin{aligned}
                & x' = x \\
                & y' = y-i \quad (i\in \{3,4,\dots, y\}) 
            \end{aligned}
            \right.
        \end{equation*}
        Since
        \begin{equation*}
            y' = y - i \le 2\left\lfloor{\frac{x}{4}}\right\rfloor + 2-3 = 2\left\lfloor{\frac{x}{4}}\right\rfloor -1 < 2\left\lfloor{\frac{x'}{4}}\right\rfloor,
        \end{equation*}
        we have $(x',y')\not \in S_{\mathcal{P}}$.
    \end{itemize}

    For (\RMii), we take $(x, y)\not\in S_{\mathcal{P}}$.
    \begin{itemize}
        \item If $y< 2\left\lfloor{\frac{x}{4}}\right\rfloor$, let $p$ be the quotient and $q$ be the remainder when $y$ is divided by $2$. We show that we can make a move $(x,y)\to (4p, y+1)$, and $(4p,y+1)\in S_{\mathcal{P}}$. First, $2\left\lfloor{\frac{4p}{4}}\right\rfloor = 2p\le y$ 
        holds and from the assumption, $2\left\lfloor{\frac{4p}{4}}\right\rfloor\le y<2\left\lfloor{\frac{x}{4}}\right\rfloor$ holds. It implies $p+1\le \left\lfloor{\frac{x}{4}}\right\rfloor$. Therefore, $4\le x - 4p$ holds and demonstrate that this move is possible. Also, since $y=2p+q$,
        \begin{equation*}
            y+1 = 2p+q+1\le 2p+2 = 2\left\lfloor{\frac{4p}{4}}\right\rfloor +2
        \end{equation*}
        holds, which shows that $2\left\lfloor{\frac{4p}{4}}\right\rfloor\le y+1 \le 2\left\lfloor{\frac{4p}{4}}\right\rfloor +2$.
        \item If $2\left\lfloor{\frac{x}{4}}\right\rfloor + 2 < y$, we show that we can make a move $(x,y)\to (x, 2\left\lfloor{\frac{x}{4}}\right\rfloor )$, and $(x, 2\left\lfloor{\frac{x}{4}}\right\rfloor )\in S_{\mathcal{P}}$.
       From the assumption, $2 < y- 2\left\lfloor{\frac{x}{4}}\right\rfloor$, which shows that this move is possible. Also, it is clear that $(x, 2\left\lfloor{\frac{x}{4}}\right\rfloor )\in S_{\mathcal{P}}$.
    \end{itemize}
Therefore, the set of \Pps~of this combinatorial game is shown by $S_{\mathcal{P}}$ and by using Theorem \ref{Thm:DYNeven}, the set of \Pps~of \dyn~on $G_5$ is $S=S_1\cup S_2$.
\end{proof}


\begin{corollary}
Let $G_6=(V,E)$ be a digraph with $V=\{v,w_1,w_2, \cdots, w_n\}$ and $E=\{(v,w_k) \mid k=1,2,\cdots,n\}$.
We consider \dyn~on $G_6$, in other words, we consider a ruleset $\Gamma=(M,f)$ such that
  \begin{eqnarray*}
  M &=& (\Zz)^{n+1}~(n\in\N),\\
  f(x,y_1,y_2,\cdots,y_{n-1},y_n)&=& \{(x-i,y_1+1,y_2+1\cdots,y_n+1) \mid n+1 \leq i \leq x \} \\
           &\cup& \{(x,y_1-i,y_2,\cdots,y_{n-1},y_n) \mid 1 \leq i \leq y_1 \} \\
           &\cup& \{(x,y_1,y_2-i,\cdots,y_{n-1},y_n) \mid 1 \leq i \leq y_2 \} \\
           &\cup& \cdots\\
           &\cup& \{(x,y_1,y_2,\cdots,y_{n-1}-i,y_n) \mid 1 \leq i \leq y_{n-1} \} \\
           &\cup& \{(x,y_1,y_2,\cdots,y_{n-1},y_n-i) \mid 1 \leq i \leq y_n \},
  \end{eqnarray*}
  where $x$ is the number of tokens on vertex $v$ and $y_j$ is that on vertex $w_j$.
  
  If $n$ is odd, then the set of \Pps~of this ruleset is
$$S = \left\{(x,y_1,\cdots,y_n)\middle| \bigoplus_{j=1}^{n} y_j =0 \right\},$$
and if $n$ is even, then the set of \Pps~of this ruleset is $S = S_1 \cup S_2$, where
\begin{eqnarray*}
\label{Eqn:f_functionpart1Pps}
S_1&=&
\left\{ 
(x,y_1,\cdots,y_n)
\middle|
\begin{gathered} 
x\leq n \\ 
\bigoplus_{j=1}^{n} y_j =0\\
\end{gathered}
\right\},\\
S_2&=&
\left\{ 
(x,y_1,\cdots,y_n)
\middle|
\begin{gathered} 
x\geq n+1 \\   
\bigoplus_{j=1}^{n} y_j =1\\
\end{gathered}
\right\}.\\
\end{eqnarray*}
\end{corollary}

Figure \ref{fig:star} shows the digraph $G_6$.

\begin{figure}[tb]
    \centering
    \includegraphics[width=1\linewidth]{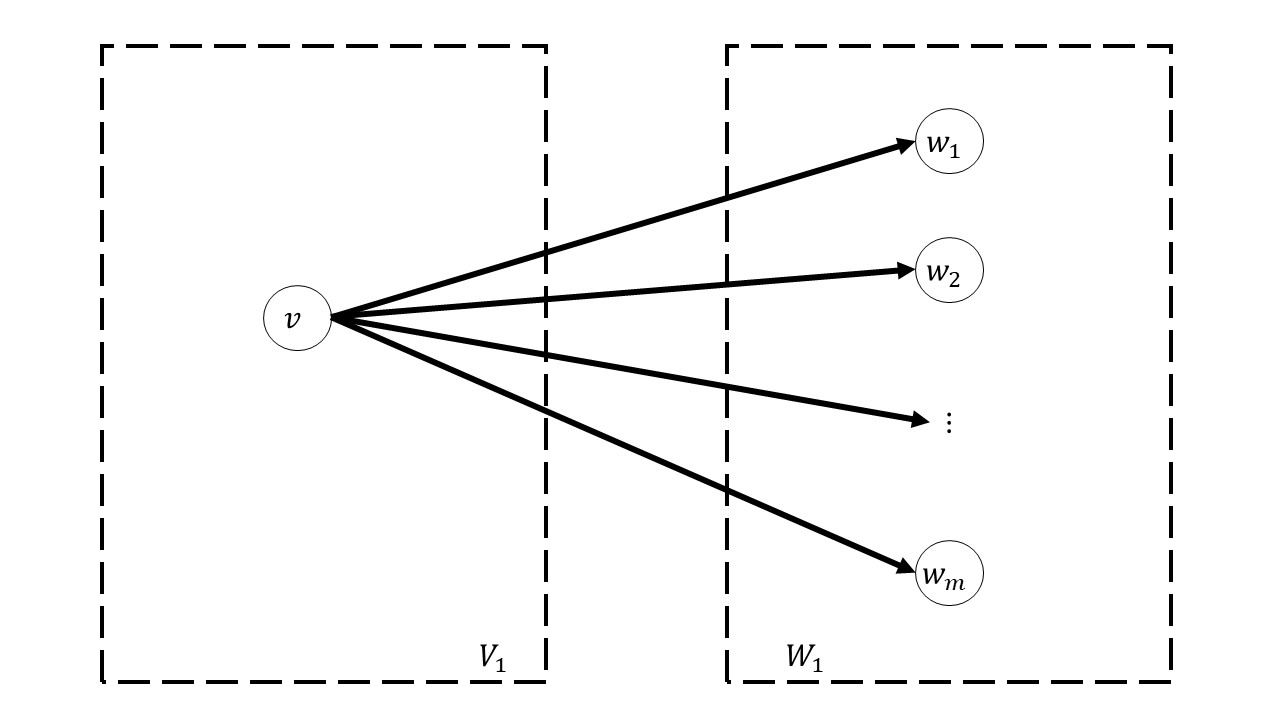}
    \caption{Digraph $G_6$}
    \label{fig:star}
\end{figure}

\begin{proof}
   When $n$ is odd, from Theorem \ref{Thm:DYNodd}, the set of \Pps~is $$S = \left\{(x,y_1,\cdots,y_n)\middle| \bigoplus_{j=1}^{n} y_j =0 \right\}.$$

   When $n$ is even, in order to apply Theorem \ref{Thm:DYNeven}, we consider a combinatorial game on $G_6$ allowing the following move.
   \begin{itemize}
       \item Remove at least $n+1$ tokens from $v$ and add $1$ token to each $w_j$.
   \end{itemize}
   Then, it is obvious that the sets $S_\mathcal{P}$ and $S_\mathcal{N}$ are 
   \begin{eqnarray*}
    S_\mathcal{P} &=& \{(x, y_1, \ldots, y_n)\mid x \le n\}, \\
    S_\mathcal{N} &=& \{(x, y_1, \ldots,y_n) \mid x \ge n+1\}.
   \end{eqnarray*}
   Therefore, the set of \Pps~of the original game is $S = S_1 \cup S_2$, where
\begin{eqnarray*}
\label{Eqn:f_functionpart1Pps}
S_1&=&
\left\{ 
(x,y_1,\cdots,y_n)
\middle|
\begin{gathered} 
x\leq n \\ 
\bigoplus_{j=1}^{n} y_j =0\\
\end{gathered}
\right\},\\
S_2&=&
\left\{ 
(x,y_1,\cdots,y_n)
\middle|
\begin{gathered} 
x\geq n+1 \\   
\bigoplus_{j=1}^{n} y_j =1\\
\end{gathered}
\right\}.\\
\end{eqnarray*}
\end{proof}

\vspace{3mm}

\bmhead{Acknowledgments}
 This work is supported by JSPS Kakenhi 23K03071. This work is also supported by JST, the establishment of university fellowships towards the creation of science technology innovation, Grant Number JPMJFS2129 and JST SPRING, Grant Number JPMJSP2132. This work is also supported by JSPS KAKENHI Grant Number JP24KJ1232.

\end{document}